\documentclass{article}

\usepackage{amsmath,amsfonts,amsthm,amssymb,amscd,color,xcolor,mathrsfs}

\usepackage{hyperref}
\usepackage{dsfont}

\colorlet{darkblue}{blue!50!black}

\hypersetup{
    colorlinks,%
    citecolor=darkblue,%
    filecolor=red,%
    linkcolor=darkblue,%
    urlcolor=magenta,%
    pdfnewwindow=true,%
    pdfstartview={FitH}
}

\usepackage[latin1,applemac,koi8-r]{inputenc}
\usepackage[cyr]{aeguill}

\newcommand{\p}{\partial}
\newcommand{\e}{\varepsilon}

\newcommand{\R}{{\mathbb R}}
\newcommand{\IP}{{\mathbb P}}

\newcommand{\Z}{{\mathbb Z}}
\newcommand{\E}{{\mathbb E}}
\newcommand{\T}{{\mathbb T}}

\newcommand{\dd}{{\textup d}}

\newcommand{\BB}{{\cal B}}

\newcommand{\DD}{{\cal D}}

\newcommand{\FF}{{\cal F}}

\newcommand{\HH}{{\cal H}}

\newcommand{\NN}{{\cal N}}

\newcommand{\PP}{{\cal P}}

\newcommand{\PPPP}{{\mathfrak P}}

\newcommand{\sgn}{\mathop{\rm sgn}\nolimits}

\newcommand{\diver}{\mathop{\rm div}\nolimits}

\theoremstyle{plain}
\newtheorem*{mt}{Main Theorem}

\newtheorem{theorem}{Theorem}[section]

\newtheorem{proposition}[theorem]{Proposition}

\theoremstyle{definition}
\newtheorem{definition}[theorem]{Definition}

\theoremstyle{remark}
\newtheorem{remark}[theorem]{Remark}

\newtheorem*{example*}{Example}

\numberwithin{equation}{section}

\begin{document}
\author{Thi Hien Nguyen\footnote{Laboratoire AGM, UMR CNRS 8088, CY Cergy Paris University, France;  e-mail: \href{mailto:Thi-Hien.Nguyen5@cyu.fr}{Thi-Hien.Nguyen5@cyu.fr}}
\and  Armen Shirikyan\footnote{Department of Mathematics, CY Cergy Paris University, CNRS UMR 8088, 2 avenue Adolphe Chauvin, 95302 Cergy--Pontoise, France;
 e-mail: \href{mailto:Armen.Shirikyan@cyu.fr}{Armen.Shirikyan@cyu.fr}}}
\title{Viscosity estimation for 2D pipe flows~I. Construction, consistency, asymptotic normality}
\date{}
\maketitle
\begin{abstract}
We consider the motion of incompressible viscous fluid in a rectangle, imposing the periodicity condition in one direction and the no-slip boundary condition in the other. Assuming that the flow is subject to an external random force, white in time and regular in space, we construct an estimator~$\hat\nu_t$ for the viscosity~$\nu$ using only observations of the~$L^2$ norm of the vorticity on the time interval $[0,t]$. The goal of the paper is to investigate the asymptotic properties of~$\hat\nu_t$ as $t\to+\infty$. It is proved that the estimator~$\hat\nu_t$ is  strongly consistent and asymptotically normal. The proof of consistency is based on the explicit formula for the estimator and some bounds for trajectories, while that of asymptotic normality uses in addition mixing properties of the Navier--Stokes flow.


\end{abstract}

\tableofcontents

\section{Introduction}
\label{s0}
The problem of estimation of viscosity in a fluid flow arises in various applications. For instance, the paper ~\cite{ZYZM-2017} describes the importance of calculating viscosity in oil-water emulsion widely used in crude oil production and transportation. In~\cite{DIG-2016}, the authors argue that a real-time control of viscosity may be used for increasing the efficiency of combustion in a fuel power plant. There are many types of devices for measuring the viscosity of a fluid, and we refer the reader to the papers~\cite{BS-2003,BKES-2018} that describe some of them used in the oil industry. The aim of this paper is to  justify theoretically the possibility of measuring fluid viscosity based on  observations of a one-dimensional  functional (not using the temperature and pressure) on relatively short time intervals. A natural mathematical model for this type of problem would have been the  three-dimensional Navier--Stokes system in a cylindrical domain, with the hypothesis that the flow stabilises to a constant velocity at infinity. However, mathematical investigation of this problem is highly complicated due to the absence of global well-posedness of the PDE in question and the unboundedness of the physical domain. Hence, we make two simplifications: the problem is studied in a two-dimensional strip, assuming the periodicity in the unbounded direction. 

We thus consider the 2D Navier--Stokes system in the domain
$$
D=\{x=(x_1,x_2)\in \R^2: -1<x_2<1\},
$$ 
imposing the no-slip condition on the boundary and the periodicity condition in the horizontal direction:
\begin{align}
	\p_tu+\langle u,\nabla\rangle u-\nu\Delta u+\nabla p&=\eta(t,x), \quad \diver u=0,\label{NS}\\
	u\big|_{x_2=\pm1}&=0, \quad \theta_au\equiv u.\label{BC}
\end{align}
Here $\nu>0$ is the viscosity, $(u,p)$ are the unknown velocity and pressure, $\eta$ is an external force whose exact form is specified below, and $\theta_a$ is a translation operator in~$x_1$ taking a function $v(x_1,x_2)$ to $v(x_1+a,x_2)$. In what follows, with a slight abuse of notation, we write~$D$ for the direct product of the circle $\T_a:=\R/(a\Z)$ and the interval~$(-1,1)$, so that any function defined on~$D$ is $a$-periodic with respect to~$x_1$. Concerning the right-hand side~$\eta$, we assume that it is a spatially regular white noise:
\begin{equation}\label{external-force}
	\eta(t,x)=\frac{\p}{\p t}\zeta(t,x), \quad \zeta(t,x)=\sum_{j=1}^\infty b_j\beta_j(t)e_j(x), 
\end{equation}
where $\{e_j\}$ is an orthonormal basis consisting of the eigenfunctions of the Stokes operator in~$D$ (see Section~\ref{s-formulations} for details), $\{\beta_j\}$ are independent standard Brownian motions, and~$\{b_j\}$ is a sequence of non-negative numbers going to zero sufficiently fast. We denote 
\begin{equation}\label{number-B}
	B:=\sum_{j=1}^\infty b_j^2
\end{equation}
and emphasise that the number~$B$ can be arbitrarily small, so that there is no restriction on the average size of the noise in~\eqref{NS}. 

It is well known that problem~\eqref{NS}, \eqref{external-force} generates a Feller-continuous Markov process in an appropriate function space (see Section~15.2 in~\cite{DZ1996} or Section~2.4 in~\cite{KS-book}), and we denote by~$u(t)$ its trajectories, which are vector functions with two components~$(u_1,u_2)$. For any $t>0$, let us set 
\begin{equation}\label{estimator}
	\hat\nu_t=\frac{Bt}{2}\biggl(\int_0^t\|\nabla u(s)\|^2\dd s \biggr)^{-1}, 
\end{equation} 
where $\nabla u$ stands for the $2\times 2$ matrix $(\p_iu_j)$, and $\|\cdot\|$ denotes the $L^2$ norm over~$D$. The following theorem is the main result of this paper; see Section~\ref{s-formulations} for an exact formulation. 

\begin{mt}
Let $B>0$. Then the properties below hold for trajectories of the stochastic Navier--Stokes system~\eqref{NS}, \eqref{BC} with arbitrary~$\nu>0$ and~$a>0$. 
\begin{description}
	\item [\sl Correctness.] The estimator $\hat \nu_t$ is well defined for almost every realisation of the noise. 
	\item [\sl Consistency.] For any $\e>0$, the difference $\hat\nu_t-\nu$ almost surely converges to zero  with the rate $t^{-\frac12+\e}$. 
	\item [\sl Asymptotic normality.] Suppose, in addition, that $b_j>0$ for all $j\ge1$. Then, for any $\e>0$, the  random variables $\sqrt{t}\,(\hat\nu_t-\nu)$  converge  weakly, with the rate~$t^{-\frac14+\e}$, to the centred normal law with some variance $\sigma_\nu^2>0$. 
\end{description}
\end{mt}

Let us mention that the problem of estimation of a drift parameter in finite-dimensional diffusion processes is rather well understood. We refer the reader to Chapter~III in~\cite{IH1981} and Chapter~2 in~\cite{kutoyants2004} for various results on this subject. The parameter estimation for infinite-dimensional diffusions usually requires rather strong hypotheses on the diffusion operator. In the case of linear SPDEs, this type of problems were studied in the papers~\cite{HKR-1993,HR-1995,LR-1999} (see also the book~\cite{RL2017} for more references), which are devoted to the investigation of MLEs based on finite-dimensional observations. Roughly speaking, those papers establish the consistency and asymptotic normality of the estimator as the dimension of observations goes to infinity. The case of the Navier--Stokes system and some other parabolic SPDEs were studied in~\cite{CGH-2011,cialenco-2018,PS-2020}. Using again finite-dimensional observations of solutions on a fixed time interval, some MLE-type estimators are constructed and their consistency and asymptotic normality are  proved as the dimension goes to infinity. The setting of this paper differs from that used in the above-mentioned works in two respects dictated by the applications discussed at the beginning of the introduction: first, we are interested in a viscosity estimator based on a {\it fixed scalar\/} observable and, second, we deal with a noise that is {\it regular\/} in the space variables. The latter makes it impossible to use the MLE-type estimators because the measures arising in the space of trajectories for different values of viscosity are likely to be singular and the very definition of MLE is far from being simple.  

In conclusion, we note that a natural question in the context of the viscosity estimation is a detailed description of probabilities of deviations of the estimator~$\hat\nu_t$ from its limiting value~$\nu$. That problem requires somewhat different techniques coming from the theory of large deviations and will be addressed in the second part of this project. 

\smallskip
The paper is organised as follows. In Section~\ref{s-MR}, we give an exact formulation of the main results and outline the scheme of their proofs. Section~\ref{s-lln} is devoted to the proof of the (strong) consistency of~$\hat\nu_t$. In Section~\ref{s-clt}, we establish the asymptotic normality of~$\hat\nu_t$. 

\subsubsection*{Acknowledgments}
This research was supported by the \textit{CY Initiative of Excellence\/} through the grant {\it Investissements d'Avenir\/} ANR-16-IDEX-0008. The first  author was supported by the project {\it EcoDep\/} PSI-AAP2020-0000000013.

\subsubsection*{Notation}
As was mentioned before, we write $D=\T_a\times(-1,1)$, where $\T_a$ is a circle of length~$a$, and assume that all the functions defined on~$D$ are $a$-periodic with respect to the variable~$x_1$. Given a separable Banach space~$X$ and a closed  interval $J\subset\R$, we write $C_b(J,X)$ for the space of bounded continuous functions $f:J\to X$ and $L^p(J,X)$ for the space of Borel-measurable functions $f:J\to X$ such that 
$$
\|f\|_{L^p(J,X)}:=\biggl(\int_J\|f(t)\|_X^p\dd t\biggr)^{1/p}<\infty,
$$
with the usual modification for $p=\infty$. We often use the following function spaces arising in the theory of the Navier--Stokes system. 
\begin{itemize}
	\item $L^p(D)$ is the usual Lebesgue space over the domain~$D$. We use the same symbol for spaces of vector and scalar functions and write $(\cdot,\cdot)$ for the $L^2$ inner product. 
	\item $H^s(D)$ is the Sobolev space of order~$s\in\Z$ on~$D$. For $s\ge1$, we denote by~$H_0^s(D)$ the closure of~$C_0^\infty(D)$ in~$H^s(D)$. 
	\item $H$ is the space of divergence-free vector fields $u:D\to\R^2$ that belong to~$L^2(D)$ and whose second component vanishes on the boundary~$\p D$. 
	\item $V$ is the intersection of~$H$ and~$H_0^1(D)$. 
	\item We write $\|\cdot\|$ for the $L^2$ norm and~$\|\cdot\|_s$ for the $H^s$ norm. For all other spaces, the norm will be specified by an appropriate subscript (e.g., $\|\cdot\|_{L^p}$). 
\end{itemize}
For any $\varkappa>0$, we write $w_\varkappa:H\to\R$ for the function $w_\varkappa(u)=\exp(\varkappa\nu\|v\|^2)$. Given a Polish space~$X$, we denote by $\BB(X)$ its Borel $\sigma$-algebra and by~$\PP(X)$ the set of probability measures on~$(X,\BB(X))$. If $\mu\in\PP(X)$ and $f:X\to\R$ is a $\mu$-integrable function, then we write $\langle f,\mu\rangle$ for the integral of~$f$ against~$\mu$.

\section{Main results}
\label{s-MR}

\subsection{Setting and construction of estimator}
\label{s-setting}
Let us denote by~$H$ the space of vector fields $u=(u_1,u_2)\in L^2(D,\R^2)$ such that $\diver u=0$ in~$D$ and $u_2=0$ on the horizontal boundary of~$D$ and introduce the space $V=H\cap H_0^1$. We write $\Pi:L^2(D,\R^2)\to H$ for the orthogonal projection to~$H$. Let us recall that the {\it Stokes operator\/} defined formally by the formula $L=-\Pi\Delta$ is a positive self-adjoint operator with the domain $\DD(L):=V\cap H^2$. It is well known that~$L$ has a discrete spectrum, and we denote by~$\{e_j\}$ an orthonormal basis in~$H$ composed of the eigenfunctions of~$L$ with eigenvalues~$\{\alpha_j\}$ indexed in the increasing order. Applying~$\Pi$  to the Navier--Stokes system~\eqref{NS} and recalling that~$\eta$ has the form~\eqref{external-force}, we derive the following (nonlocal) SPDE
\begin{equation}\label{NS-reduced}
	\p_tu+\nu Lu+B(u)=\eta(t),
\end{equation}
where $B(u)=\Pi((u_1\p_1+u_2\p_2)u)$.  The Cauchy problem for~\eqref{NS-reduced} is well posed, so that for any $u_0\in H$ there is a unique solution $u(t,x)$ whose almost every trajectory belongs to the space 
\begin{equation*}
	\HH:=L_{\textrm{loc}}^2(\R_+,V)\cap C(\R_+,H)
\end{equation*}
and satisfies the initial condition
\begin{equation}\label{IC}
	u(0)=u_0.
\end{equation}
Moreover, the family of solutions corresponding to all possible initial conditions $u_0\in H$ form a Markov process, which is denoted by $(u_t,\IP_{u_0})$. Here, $u_t=u_t(\omega)$ stands for the trajectory and $\IP_{u_0}$ denotes the probability associated with the initial condition $u_0\in H$. If $\lambda\in\PP(H)$, then we denote by~$\IP_\lambda$ the probability associated with the initial measure~$\lambda$,
$$
\IP_\lambda(\cdot)=\int_H\IP_v(\cdot)\lambda(\dd v),
$$
and write~$\E_\lambda$ for the corresponding expectation. We refer the reader to the books~\cite[Chapter~15]{DZ1996} and~\cite[Chapter~2]{KS-book} for further details and proofs.

The derivation of estimator~\eqref{estimator} is based on an application of the It\^o formula. More precisely, applying the It\^o formula to the $L^2$ norm of a solution, we derive
\begin{equation}\label{ito-L2}
	\|u(t)\|^2+2\nu\int_0^t\|\nabla u(s)\|^2\dd s=\|u_0\|^2+Bt+2\int_0^t\bigl(u(s),\dd\zeta(s)\bigr),
\end{equation}
where~$B$ is defined by~\eqref{number-B}; see~\cite[Section~X.4]{VF1988} or \cite[Section~2.4]{KS-book}. Dividing both sides of~\eqref{ito-L2} by~$2\nu t$ and rearranging the terms, we obtain
\begin{equation}\label{xi-t}
\xi_t:=\frac1t\int_0^t\|\nabla u(s)\|^2\dd s
=\frac{B}{2\nu}+\frac{1}{2\nu t}\bigl(\|u_0\|^2-\|u(t)\|^2\bigr)
+\frac{1}{\nu t}\int_0^t\bigl(u(s),\dd\zeta(s)\bigr).	
\end{equation}
The right-most term in~\eqref{xi-t} is likely to converge to~$\frac{B}{2\nu}$ due to a priori estimates for solutions and law of large numbers for zero-mean martingales. It is therefore natural to define an estimator for~$\nu$ by the formula $\hat\nu_t=\frac{B}{2\xi_t}$. This expression coincides with~\eqref{estimator}. The following result shows that~$\hat\nu_t$ is well defined.

\begin{proposition}\label{p-correctness}
	Let $B>0$. Then, for any $u_0\in H$, there is a set of full measure $\Omega_*\subset\Omega$ such that, for any $\omega\in\Omega_*$, we have $\xi_t>0$ for all $t>0$. In particular, with probability~$1$, the estimator~$\hat\nu_t$ is well defined for all $t>0$. 
\end{proposition}

\begin{proof}
	Since $t\xi_t$ is an increasing function of~$t\ge0$, and $u\in\HH$ with probability~$1$, it suffices to prove that, for any $t>0$, we have 
\begin{equation}\label{P-xit}
\IP\bigl(\{\xi_t=0\}\cap\{u\in\HH\}\bigr)=0.
\end{equation}
Let $\omega\in\Omega$ be such that $u\in\HH$ and $\xi_t=0$. Then $\nabla u(s)=0$ for almost every $s\in[0,t]$. Since $u(s)\in V$ almost everywhere and $u$ is a continuous function of time with range in~$H$, we conclude that $u(s)=0$ for $s\in[0,t]$. It follows from~\eqref{NS-reduced} that $\eta(s)=0$ for $s\in[0,t]$. Since $B>0$, there is an integer $j\ge1$ such that $\beta_j(s)=0$ for $s\in[0,t]$. This event has probability zero, so that we arrive at~\eqref{P-xit}. 
\end{proof}

\subsection{Statement of the results}
\label{s-formulations}
To formulate our main results, we briefly recall some definitions. Let $\PP(\R)$ be the space of Borel probability measures on~$\R$ and let $\NN_\sigma\in\PP(\R)$ be the centered normal law with variance~$\sigma^2$. For a random variable~$\xi$, we write~$\DD(\xi)$ for its law. 

\begin{definition}
	We shall say that the estimator~$\hat\nu_t$ for~$\nu$ is:
	\begin{itemize}
	\item {\it strongly consistent\/} if
\begin{equation*}
	\IP_{u}\bigl\{\hat\nu_t\to\nu\mbox{ as }t\to\infty\bigr\}=1\quad\mbox{for any $u\in H$ and $\nu>0$};
\end{equation*}
	\item {\it asymptotically normal\/} if for any $\nu>0$ there is $\sigma_\nu>0$ such that
\begin{equation*}
	\DD\bigl(\sqrt{t}(\hat\nu_t-\nu)\bigr)
	\rightharpoonup\NN_{\sigma_\nu}
\quad\mbox{as $t\to\infty$ for any $u\in H$},
\end{equation*}
where the convergence holds in the weak topology of~$\PP(\R)$.
	\end{itemize}
\end{definition}
In what follows, we shall always assume that {\it the number~$B$ defined in~\eqref{number-B} is finite and strictly positive\/}. The following two results give the precise statement of the Main Theorem formulated in the Introduction. 

\begin{theorem}[Strong consistency]\label{thm: strong consistency}
For any $\nu\in(0,1]$ and $\e\in(0,\frac12)$, there is a random time $T\ge1$ such that 
	\begin{equation}\label{consistency hatnu}
		|\hat\nu_t - \nu| \leq t^{-\frac{1}{2}+\e} \quad \mbox{for $t\geq T$}.
	\end{equation}
Moreover, there is a number~$\varkappa>0$ not depending on~$\nu$ and~$\e$ such that, for any integer $m\ge1$ and any initial measure $\lambda\in\PP(H)$, we have
\begin{equation}\label{moments-T}
	\E_\lambda T^m\le C_m\langle w_\varkappa,\lambda\rangle,
\end{equation}
where $w_\varkappa(v)=\exp(\varkappa\nu\|v\|^2)$, and $C_m>0$ is a number depending on~$\nu$ and~$\e$. In particular, the estimator~$\hat\nu_t$ is strongly consistent. 
\end{theorem}

Let us  emphasise that, in Theorem~\ref{thm: strong consistency}, we only assume that $B>0$, without requiring the positivity of all the coefficients~$b_j$. In particular, the Markov process associated with~\eqref{NS-reduced} does not need to be mixing. The latter property is, however, important in the proof of asymptotic normality. Let us denote by~$\varPhi_\sigma$ the distribution function of the centred normal law with variance~$\sigma^2>0$. 

\begin{theorem}[Asymptotic normality]\label{thm:asympto}
In addition to the above hypotheses, let us assume that $b_j>0$ for any $j\ge1$. Then there is a constant $\sigma_\nu>0$ depending only on $\nu\in(0,1]$ such that, for any number $\e>0$ and an appropriate increasing function $C_\e:\R_+\to\R_+$,  we have 
	\begin{equation}\label{AN-with-rate}
	\sup_{z\in\R}\,
	\bigl|\IP_\lambda\bigl\{\sqrt{t}\,\bigl(\hat\nu_t-\nu\bigr)\le z\bigr\}-\varPhi_{\sigma_\nu}(z)\bigr|\le C_\e\bigl(\langle w_\varkappa,\lambda\rangle\bigr)\,t^{-\frac14+\e},
	\end{equation}
	where $t\ge1$ is arbitrary, $\lambda\in\PP(H)$ is the law of the initial state, and the number $\varkappa>0$ does not depend on~$\nu$, $\lambda$, and~$\e$. In particular, the estimator~$\hat\nu_t$ is asymptotically normal. 
\end{theorem}

These two theorems are established in Sections~\ref{s-lln} and~\ref{s-clt}. Here we present the main ideas of their proofs.

\subsection{Scheme of the proofs}
\label{s-scheme}

\subsubsection*{Consistency}
Let us recall that the random process~$\xi_t$ is defined by~\eqref{xi-t}, so that $\hat\nu_t=\frac{B}{2\xi_t}$. It follows that 
\begin{equation}\label{nut-nu}
	|\hat\nu_t-\nu|=\frac{\nu}{\xi_t}\,\biggl|\xi_t-\frac{B}{2\nu}\biggr|.
\end{equation}
We see that the required convergence will be established if we prove that the expression $|\xi_t-(2\nu)^{-1}B|$ does not exceed~$t^{-\frac12+\e}$ for $t\ge T$, where $T=T_{\e,\nu}>0$ is a random time all of whose moments are finite. Setting $M_t=\int_0^t(u(s),\dd\zeta(s))$, we can write
\begin{equation*}
	\biggl|\xi_t-\frac{B}{2\nu}\biggr|
	\le \frac{1}{2\nu t}\|u_0\|^2+\frac{1}{2\nu t}\|u(t)\|^2+\frac{1}{\nu t}|M_t|.
\end{equation*}
The fact that the first two terms are bounded by~$t^{-\frac12+\e}$ follows from a priori estimates for solutions. The last term will be treated with the help of a technique based on estimates of moments; cf.~\cite[Section~12]{lamperti1996}.

\subsubsection*{Asymptotic normality}
A well-known consequence of Slutsky's theorem (see~\cite{JP2003}) is that the property of asymptotic normality is preserved under smooth maps. In Section~\ref{s-slutsky}, we prove a version of that result with an explicit estimate for the rate of convergence. This and the explicit form of the estimator~\eqref{estimator} reduce~\eqref{AN-with-rate} to the proof of the inequality 
\begin{equation}\label{AN-reduced}
	\sup_{z\in\R}\,
	\bigl|\IP_\lambda\bigl\{\sqrt{t}\,\bigl(\xi_t-(2\nu)^{-1}B\bigr)\le z\bigr\}-\varPhi_{\sigma_\nu}(z)\bigr|\le \widetilde C_\e\bigl(\langle w_\varkappa,\lambda\rangle\bigr)\,t^{-\frac14+\e},
\end{equation}
where $\widetilde C_\e:\R_+\to\R_+$ is an increasing function. Recalling~\eqref{xi-t}, we see that $\xi_t-(2\nu)^{-1}B$ differs from the time-average of a martingale by a negligible term. Various results on the CLT for discrete-time martingales are presented in~\cite{HH1980}. In particular, as is proved in~\cite[Section~3.5]{HH1980}, the CLT for martingales can be reduced to the law of large numbers for the conditional variance. The latter will be established with the help of the mixing property of the flow.

\section{Consistency}
\label{s-lln}
In this section, we establish the strong consistency of the estimator~$\hat\nu_t$. To this end, we shall need two auxiliary results that are discussed in the next subsection. We always assume that the hypotheses of Theorem~\ref{thm: strong consistency} are fulfilled. 

\subsection{Auxiliary results}
\label{s-consistency-auxiliry}
In what follows, we assume that the numbers~$b_j$ entering~\eqref{external-force} are fixed and do not follow the dependence of unessential constants on them. Let us denote by $\alpha_1>0$ the smallest eigenvalue of the Stokes operator~$L$ and define 
$$
\gamma=\frac14\alpha_1\Bigl(\,\sup_{j\ge1}b_j^2\Bigr)^{-1}.
$$
The following result follows from Proposition~2.4.10 in~\cite{KS-book}. 

\begin{proposition}\label{p-supermartingale}
	For any numbers $\nu\in(0,1]$, $\rho>0$ and any $H$-valued random initial condition~$u_0$ such that $\E\,\|u_0\|^2<\infty$, the solution of problem~\eqref{NS-reduced}, \eqref{IC} satisfies the inequality 
\begin{equation}\label{supermartingale}
\IP\bigl\{\|u(t)\|^2\le \|u_0\|^2 +Bt+\rho\mbox{ for }t\ge0\bigr\}\ge 1-e^{-\gamma\nu\rho}. 
\end{equation}
\end{proposition}

To formulate the second result, let us recall that $M_t$ stands for the martingale defined by the last term in~\eqref{xi-t}. 

\begin{proposition}\label{p-martingale}
There is a number $\varkappa>0$ not depending on~$\nu\in(0,1]$ such that, for any  $p\ge1$, one can find a number $C_p>0$ with the following property: if an $H$-valued initial condition~$u_0$ is such that $\E \exp(\varkappa\nu\|u_0\|^2)<\infty$, then for any $t\ge0$ the martingale~$M_t$ satisfies the inequalities
	\begin{align}
		\E\,\Bigl(\,\sup_{0\le s\le t}|M_s|^{2p}\Bigr)&\le C_p\,t^p\bigl(1+t^{-1}\,\E\,e^{\varkappa\nu\|u_0\|^2}\bigr), \label{Mt-moments}\\
		\E\,\Bigl(\,\sup_{t\le s\le t+1}|M_s-M_t|^{2p}\Bigr)&\le C_p\bigl(1+e^{-\varkappa\nu^2 t}\,\E\,e^{\varkappa\nu\|u_0\|^2}\bigr).\label{Mt-Ms}
	\end{align} 
\end{proposition}

\begin{proof}
In view of the Burkholder--Davis--Gundy (BDG) inequality (see~\cite{BDG-1972} or~\cite[Theorem~3.28]{KS1991}), for any~$p\ge1$, we have
	\begin{equation*}
		\E\,\Bigl(\,\sup_{0\le s\le t}|M_s|^{2p}\Bigr)
		\le c_1\,\E\biggl(\int_0^t\sum_{j=1}^\infty b_j^2u_j^2(s)\,\dd s\biggr)^p,
	\end{equation*}
	where $u_j=(u,e_j)$ is the $j^\text{\rm th}$ component of the solution~$u$, and~$c_i>0$ stand for numbers that may depend only on~$p$ and~$\nu$. Estimating~$b_j^2$ by~$B$ and using the H\"older inequality, we derive 
	\begin{equation}\label{BDG-M-t}
	\E\,\Bigl(\,\sup_{0\le s\le t}|M_s|^{2p}\Bigr)
	\le c_2\,t^{p-1}\int_0^t\E\,\|u(s)\|^{2p}\,\dd s.
	\end{equation}
	Now let $\varkappa>0$ be the number from Proposition~2.4.9 in~\cite{KS-book}, so that 
	\begin{equation}\label{expo-estimate}
		\E\,e^{\varkappa\nu\|u(s)\|^2}\le e^{-\varkappa\nu^2 s}\,\E\,e^{\varkappa\nu\|u_0\|^2}+c_3, \quad s\ge0. 
	\end{equation}
	Using the inequality $y^p\le c_4e^{\varkappa\nu y}$ with $y=\|u(s)\|^2$ and combining inequalities~\eqref{BDG-M-t} and~\eqref{expo-estimate}, we arrive at~\eqref{Mt-moments}. 
	
	To prove~\eqref{Mt-Ms}, we use again the BDG inequality. Arguing as in the case of~\eqref{Mt-moments}, we derive
	$$
	\E\,\Bigl(\,\sup_{t\le s\le t+1}|M_s-M_t|^{2p}\Bigr)
	\le c_5\int_t^{t+1}\E\,\|u(s)\|^{2p}\dd s
	\le c_6\int_t^{t+1}\E\,e^{\varkappa\nu\|u(s)\|^2}\dd s. 
	$$
	Using~\eqref{expo-estimate} to estimate the integrand in the right-most term of this inequality, we obtain~\eqref{Mt-Ms}. 
\end{proof}

\subsection{Proof of Theorem~\ref{thm: strong consistency}}
Let us recall that the Navier--Stokes system~\eqref{NS-reduced} generates a Markov process~$(u_t,\IP_{u_0})$ and denote by~$\{\FF_t\}_{t\ge0}$ the associated filtration. The proof is divided into three steps. 

\smallskip
{\it Step~1: Reduction to~$\xi_t$\/}. Suppose we have proved that there is a number $\varkappa>0$ such that, for any $\e\in(0,\frac12)$ and $\nu\in(0,1]$, one can find a random time $T_0\ge1$ for which 
\begin{align}
	\bigl|\xi_t-(2\nu)^{-1}B\bigr|
	&\le t^{-\frac12+\e}\quad\mbox{for $t\ge T_0$},\label{xi-t-nuB}\\
	\E_\lambda T_0^m&\le C_m\int_He^{\varkappa\nu\|u\|^2}\lambda(\dd u),\label{El-T0}
\end{align}
where $\lambda\in\PP(H)$ is an arbitrary initial measure, $m\ge1$ is any integer, and $C_m>0$ is a number depending on~$\e$ and~$\nu$. In this case, taking $\e=\frac14$ in~\eqref{xi-t-nuB}, we see that $|\xi_t|\ge (4\nu)^{-1}B$ for $t\ge T_0':=T_0\vee (4B^{-1})^4$. Combining this with relation~\eqref{nut-nu} and inequality~\eqref{xi-t-nuB}, in which~$\e$ is replaced with~$\e/2$, we obtain
$$
|\hat\nu_t-\nu|\le 4B^{-1}t^{-\frac12+\frac\e2}\quad\mbox{for $t\ge T_0'$}. 
$$
The right-hand side of this inequality does not exceed $t^{-\frac12+\e}$, provided that $t\ge (4B^{-1})^{2/\e}$. We thus obtain inequality~\eqref{consistency hatnu}, in which $T=T_0'\vee(4B^{-1})^{2/\e}$. The validity of~\eqref{moments-T} follows from a similar inequality for~$T_0$. 

To prove~\eqref{xi-t-nuB} and~\eqref{El-T0}, it suffices to construct random times~$T_1, T_2, T_3\ge1$ such that~\eqref{El-T0} holds for each of them, and
\begin{align}
	\qquad (2\nu t)^{-1}\|u_0\|^2&\le \tfrac13\,t^{-1/2}&\quad&\mbox{for $t\ge T_1$},\qquad\label{u0t}\\
	\qquad(2\nu t)^{-1}\|u(t)\|^2&\le \tfrac13\,t^{-1/2}&\quad&\mbox{for $t\ge T_2$},\qquad\label{utt}\\
	\qquad(\nu t)^{-1}|M_t|&\le \tfrac13\,t^{-1/2+\e}&\quad&\mbox{for $t\ge T_3$}.\qquad\label{Mtt}
\end{align}
This will be done in the next two steps. 

\smallskip
{\it Step~2: Estimates for~$\|u_0\|^2$ and~$\|u(t)\|^2$}. Inequality~\eqref{u0t} is satisfied with $T_1=\bigl(\frac{3}{2\nu}\bigr)^2\|u_0\|^4$. The validity of~\eqref{El-T0} for~$T_1$ is obvious. To prove~\eqref{utt}, we introduce the random variables 
$$
U_k:=\sup_{k\le t\le k+1}\|u(t)\|^2, \quad k\ge0,
$$
and define the events $A_k=\{U_k>\frac{2\nu}{3}k^{1/2}\}$; cf.~\eqref{utt}. Suppose we found positive numbers $C$, $c$, and~$\gamma$ such that, for any initial measure $\lambda\in\PP(H)$,
\begin{equation}\label{P-Ak}
	\IP_\lambda(A_k)\le C\,e^{-ck^{1/2}}\langle w_\gamma,\lambda\rangle,
	\quad k\ge1.
\end{equation}
In this case, by the Borel--Cantelli lemma, the random time
$$
T_2:=\min\bigl\{N\ge1:U_k\le \tfrac{2\nu}{3}\,k^{1/2}\mbox{ for }k\ge N\bigr\}
$$
is $\IP_\lambda$-almost surely finite. For $k\ge T_2$ and $t\in[k,k+1]$, we derive 
$$
\|u(t)\|^2\le U_k\le \frac{2\nu}{3}\,k^{1/2}\le \frac{2\nu}{3}t^{1/2}.
$$
We thus obtain the validity of~\eqref{utt}. Moreover, for any integer $m\ge1$, we can write 
$$
\E_\lambda T_2^m=\sum_{k=1}^\infty k^m\,\IP_\lambda\{T_2=k\}
\le 1+\sum_{k=2}^\infty k^m\,\IP_\lambda(A_{k-1}). 
$$
Combining this with~\eqref{P-Ak}, we see that~\eqref{El-T0} holds for~$T_2$ with $\varkappa=\gamma$. Thus, it remains to prove~\eqref{P-Ak}.

To this end, we use inequality~\eqref{supermartingale} and the Markov property. Namely, applying~\eqref{supermartingale} to a deterministic initial condition~$v\in H$, we see that 
$$
\IP_v\{U_0>r\}\le \exp\bigl\{-\gamma\nu(r-B-\|v\|^2)\bigr\}\quad\mbox{for any $r>0$}.
$$
Taking $r=\frac{2\nu}{3}k^{1/2}$ and assuming without loss of generality that $\gamma\le\varkappa$, where $\varkappa>0$ is the number in~\eqref{expo-estimate}, we derive
\begin{align*}
	\IP_\lambda(A_k)
	&=\E_\lambda\bigl(\IP_\lambda\{A_k\,|\,\FF_k\}\bigr)=\E_\lambda\IP_{u_k}\bigl\{U_0>\tfrac{2\nu}{3}\,k^{1/2}\bigr\}
	\le C\,e^{-ck^{1/2}}\,\E_\lambda e^{\gamma\nu\|u_k\|^2}.
\end{align*}
Combining this with~\eqref{expo-estimate}, at arrive at~\eqref{P-Ak}.  

\smallskip
{\it Step~3: Estimate for~$M_t$}. Let us set 
$$
V_k:=\sup_{k\le t\le k+1}|M_t-M_k|, \quad k\ge0.
$$
To prove~\eqref{Mtt}, it suffices to construct a random integer $T_3\ge1$ satisfying~\eqref{El-T0} such that 
\begin{equation}\label{Mk-Vk-bound}
	|M_k|\le\frac{\nu}{10}\,k^{\frac12+\e}, \quad 
	V_k\le\frac{\nu}{10}\,k^{\frac12+\e}\quad\mbox{for $k\ge T_3$}. 
\end{equation}
Indeed, if these inequalities are established, then taking any $t\ge T_3$ and denoting by~$k\le t$ the integer part of~$t$, we can write
$$
t^{-1}|M_t|\le t^{-1}|M_t-M_k|+t^{-1}|M_k|
\le k^{-1}(V_k+|M_k|)\le \frac{\nu}{5}\,k^{-\frac12+\e}\le \frac{\nu}{3}\,t^{-\frac12+\e}. 
$$
This coincides with~\eqref{Mtt}. The proofs of inequalities~\eqref{Mk-Vk-bound} are based on~\eqref{Mt-moments} and~\eqref{Mt-Ms}, and the argument is similar to that used in Step~2. Therefore, we confine ourselves to the proof of the estimate for~$M_k$. 

Let us fix any $\e\in(0,\frac12)$ and define the random time 
$$
T_3=\min\{N\ge1: |M_k|\le \tfrac{\nu}{10}\,k^{\frac12+\e}\}.
$$
Using~\eqref{Mt-moments} and the Chebyshev inequality, for any $p\ge1$ we derive
\begin{equation}\label{PMk-bound}
\IP_\lambda\bigl\{|M_k|>\tfrac{\nu}{10}\,k^{\frac12+\e}\bigr\}\le \bigl(\tfrac{\nu}{10}\,k^{\frac12+\e}\bigr)^{-2p}\,\E_\lambda|M_k|^{2p}\le C_p'k^{-2p\e}\langle w_\varkappa,\lambda\rangle,	
\end{equation}
where the number~$C_p'>0$ does not depend on~$\lambda$. Thus, taking $p>(2\e)^{-1}$ and using the Borel-Cantelli lemma, we conclude that~$T_3$ is $\IP_\lambda$-almost surely finite. Moreover, in view of~\eqref{PMk-bound} and the definition of~$T_3$, for any $m\ge1$, we have
\begin{align*}
\E_\lambda T_3^m&=\sum_{k=1}^\infty k^m\,\IP_\lambda\{T_3=k\}
\le 1+\sum_{k=2}^\infty k^m\,\IP_\lambda\bigl\{|M_k|>\tfrac{\nu}{10}\,k^{\frac12+\e}\bigr\}\\
&\le 1+C_p'\langle w_\varkappa,\lambda\rangle\sum_{k=2}^\infty k^{m-2p\e}.
\end{align*}
Taking $p>\frac{m+1}{\e}$, we see that the right-most term in this inequality can be estimated from above by the right-hand side of~\eqref{El-T0}. This completes the proof of Theorem~\ref{thm: strong consistency}.\qed

\begin{remark}\label{r-growth-delta}
	Exactly the same argument as in Step~2 in the above proof enables one to derive the following sharper version of~\eqref{u0t} and~\eqref{utt}: for any $\delta>0$ there is a random time $T_0\ge1$ satisfying~\eqref{El-T0} such that 
	\begin{equation}\label{growth-delta}
	\|u(t)\|^2+\|u_0\|^2\le t^{\delta}\quad\mbox{for $t\ge T_0$}. 
	\end{equation}
	This observation will be important in the next section when proving the asymptotic normality of~$\hat\nu_t$. 
\end{remark}

\section{Asymptotic normality}
\label{s-clt}

\subsection{Slutsky theorem with rate of convergence}
\label{s-slutsky}

As in the case of the consistency, the asymptotic normality is established with the help of a reduction to the study of similar property for the random process~$\xi_t$. To this end, we shall need the following version of the $\delta$-method which is a consequence of Slutsky theorem; cf.~\cite[Chapter~18]{JP2003}.

\begin{proposition}\label{p-slutsky}
	Let $\{\eta_t\}_{t\ge0}$ be a real-valued random process such that
\begin{equation}\label{slutsky-estimate}
	\sup_{s\in\R}\,\bigl|\IP\{\sqrt{t}\,(\eta_t-a)\le s\}-\varPhi_\sigma(s)\bigr|\le C\,t^{-b}\quad \mbox{for $t\ge1$},
\end{equation}
where $a\in\R\setminus\{0\}$, $b\in(0,\frac14]$, $C>0$, and~$\sigma>0$ are some numbers. Then, for any $c>0$, we have
\begin{equation}\label{slutsky-conclusion}
	\sup_{s\in\R}\,\bigl|\IP\{c\sqrt{t}\,(\eta_t^{-1}-a^{-1})\le s\}-\varPhi_{\hat\sigma}(s)\bigr|\le L\,t^{-b}\quad \mbox{for $t\ge1$},
\end{equation}
where $\hat\sigma=c\sigma/a^2$, and $L>0$ is a number depending only on~$a$, $\sigma$,  and~$c$. 
\end{proposition}

\begin{proof}
	There is no loss of generality in assuming that $c=1$ and $a>0$. To prove~\eqref{slutsky-conclusion} with $c=1$, it suffices to establish the following upper and lower bounds for any $s\in\R$ and $t\ge1$:
	\begin{align}
		\IP\bigl\{\sqrt{t}\,(\eta_t^{-1}-a^{-1})\le s\bigr\}&\le \varPhi_{\hat\sigma}(s)+C_1t^{-b},\label{sl-upper}\\
		\IP\bigl\{\sqrt{t}\,(\eta_t^{-1}-a^{-1})\le s\bigr\}&\ge \varPhi_{\hat\sigma}(s)-C_1t^{-b}.\label{sl-lower}
	\end{align}
	We confine ourselves to the proof of~\eqref{sl-upper}, since~\eqref{sl-lower} can be established by a similar argument. 
	
In what follows, we always assume that $t\ge1$. Inequality~\eqref{slutsky-estimate} implies that, for any interval\footnote{The interval~$I$ may be bounded or unbounded and may contain or not its endpoints.} $I\subset\R$, we have
\begin{equation}\label{slutsky-interval}
	\bigl|\IP\{\sqrt{t}\,(\eta_t-a)\in I\}-\varPhi_\sigma(I)\bigr|\le 2C\,t^{-b},
\end{equation}
where $\varPhi_\sigma(I)$ is the measure of~$I$ with respect to the centred normal law~$\NN_\sigma$. Let us set $\gamma_t=\frac12at^{1/4}$ and $I_t=[-\gamma_t,\gamma_t]$. In view of~\eqref{slutsky-interval}, we have
\begin{equation}\label{xi-a}
	\bigl|\IP\bigl\{|\eta_t-a|\le \tfrac{1}{2}\,at^{-1/4}\bigr\}-\varPhi_\sigma(I_t)\bigr|
	\le 2C\,t^{-b}. 
\end{equation}
	It is straightforward to check that 
	$$
	\varPhi_\sigma(I_t)\ge 1-2\bigl(\sqrt{2\pi}\,\gamma_t\bigr)^{-1}\sigma e^{-\gamma_t^2/2\sigma^2}\ge 1-C_1t^{-b},
	$$
	where we denote by~$C_i$ some numbers depending only on~$a$, $\sigma$, and~$C$. Combining this with~\eqref{xi-a}, we see that 
\begin{equation}\label{xi-a-lb}
	\IP\bigl\{|\eta_t-a|\le \tfrac{1}{2}\,at^{-1/4}\bigr\}
	\ge 1-C_2t^{-b}. 
\end{equation}
Let us denote by~$A_t(s)$ and~$B_t$ the events under the probability signs in~\eqref{sl-upper} and~\eqref{xi-a-lb}, respectively. It follows from~\eqref{xi-a-lb} that
\begin{equation}\label{PAts}
	\IP\bigl(A_t(s)\bigr)\le \IP\bigl(A_t(s)\cap B_t\bigr)+\IP(B_t^c)
	\le \IP\bigl(A_t(s)\cap B_t\bigr)+C_2t^{-b}. 
\end{equation}
Now note that if $\omega\in A_t(s)\cap B_t$ for some $t\ge1$, then 
$$
\sqrt{t}\,(\eta_t-a)\ge -a^2s\bigl(1+\tfrac12(\sgn s)\,t^{-1/4}\bigr),
$$
where $\sgn s$ stands for the sign of~$s$. Therefore, denoting by $J_{s,t}$ the interval with the endpoints $-a^2s$ and $-a^2s-\frac12a^2|s|\,t^{-1/4}$, we can write
$$
\IP\bigl(A_t(s)\cap B_t\bigr)
\le \IP\{\sqrt{t}\,(\eta_t-a)\ge -a^2s\}
+\IP\{\sqrt{t}\,(\eta_t-a)\in J_{s,t}\}.
$$
The first term on the right-hand side can be estimated by 
$$
\varPhi_\sigma\bigl([-a^2s,+\infty)\bigr)+2C\,t^{-b}=\varPhi_{\hat\sigma}(s)+2C\,t^{-b},
$$
while the second does not exceed $C_3\,t^{-b}$. Substituting these estimates into~\eqref{PAts}, we arrive at~\eqref{sl-upper}.
\end{proof}

\subsection{Proof of Theorem~\ref{thm:asympto}}
We begin with a number of reductions. In view of Proposition~\ref{p-slutsky}, it suffices to prove inequality~\eqref{AN-reduced} with some number~$\sigma_\nu>0$. Given a random variable~$X$, let us denote $\Delta_\sigma(X,z)=F_X(z)-\varPhi_\sigma(z)$, where $F_X(z)=\IP\{X\le z\}$ is the distribution function of~$X$. Then, for any $\e>0$ and any random variables~$X$ and~$Y$, we have (e.g., see Lemma~2.9 in~\cite{shirikyan-ptrf2006}) 
	\begin{equation*}
		\sup_{z\in\R} \,|\Delta_\sigma(X,z)|\le 
		\sup_{z\in\R} \,|\Delta_\sigma(Y,z)|+ \IP \{ |X-Y|>\e\} + c_\sigma\e,
	\end{equation*}
	where $c_\sigma = (\sigma\sqrt{2\pi})^{-1}$. Recalling representation~\eqref{xi-t}, using inequality~\eqref{growth-delta}, and applying the argument in Step~1 the proof Theorem~2.8 in~\cite{shirikyan-ptrf2006}, we see that it suffices to establish the inequality 
	\begin{equation}\label{clt-for-Mk}
		\sup_{z\in\R}\,\bigl|\IP_\lambda\{k^{-1/2}M_k\le z\}-\varPhi_\sigma(z)\bigr|
		\le C_\e\bigl(\langle w_\varkappa,\lambda\rangle\bigr)\,k^{-\frac14+\e}, \quad k\ge1,
	\end{equation}
where $C_\e:\R_+\to\R_+$ is an increasing function. The proof of this inequality is based on the study of the conditional variance. 

Namely, following~\cite[Section~2.7]{HH1980}, we define the {\it conditional variance\/} for the zero-mean martingale~$M_k$ by the formula
$$
V_n^2 = \sum_{k=1}^{n} \E_\lambda \bigl\{(M_k-M_{k-1})^2\,\big|\,\mathcal{F}_{k-1}\bigr\},
$$
where $\{\FF_t\}$ is the filtration associated with the Markov process~$(u_t,\IP_{u_0})$. The following result is a modified version of Theorem~3.7 in~\cite{HH1980} and reduces the proof of~\eqref{clt-for-Mk} to the law of large numbers for~$V_n$. Its proof can be found in~\cite[Section~4.2]{shirikyan-ptrf2006}.

	\begin{proposition}\label{p-clt-lln}
	Suppose there are positive numbers~$\theta$ and~$\Theta$ such that
		\begin{equation}\label{ineq:expX_k}
		\E_\lambda \exp (\theta\,|M_k-M_{k-1}|) \le \Theta\quad\mbox{for $k\ge1$}.
		\end{equation}
		Then, for any $\bar\sigma>0$ and $\e \in (0, \frac{1}{4})$, there is a constant $A_\e(\bar\sigma)>0$ depending on~$\theta$ and~$\Theta$ such that, for any $q>0$,  $\sigma\geq\bar\sigma$, and $n\ge1$, we have
		\begin{equation}\label{clt-lln}
		\sup_{z\in\R} \bigl|\Delta_\sigma(n^{-\frac{1}{2}}M_n,z)\bigr| \leq A_\e(\bar{\sigma}) n^{-\frac{1}{4}+\e} + \sigma^{-4q} n^{q(1-4\e)} \,\E_\lambda\bigl|n^{-1}V_n^2 - \sigma^2\bigr|^{2q}.
		\end{equation}
	\end{proposition}

	Thus, to prove~\eqref{clt-for-Mk}, it suffices to establish~\eqref{ineq:expX_k} and to estimate the second term on the right-hand side of~\eqref{clt-lln}. This is done in the next two steps. 
	
	\smallskip
	{\it Step~1: Proof of~\eqref{ineq:expX_k}\/}. Let us define the martingales
	$$
	X_t(k)=M_t-M_{k-1}=\int_{k-1}^t\bigl(u(s),\dd\zeta(s)\bigr), \quad t\in[k-1,k], \quad k\ge1.
	$$
	In view of the inequality $e^{|y|}\le e^y +e^{-y}$, to prove~\eqref{ineq:expX_k}, it suffices to estimate the quantities $\E_\lambda e^{\theta X_k}$ for $|\theta|\ll1$, where $X_k=X_k(k)$. To this end, we first note that the quadratic variation of~$X_t(k)$ on the interval~$[k-1,k]$ can be written as 
	$$
	\langle X(k)\rangle
	=\sum_{j=1}^\infty\int_{k-1}^kb_j^2u_j^2(t)\,\dd t
	\le B\int_{k-1}^k\|u(t)\|^2\dd t,
	$$ 
	where $u_j(t)=(u(t),e_j)$. By the Cauchy--Schwarz inequality, for any $\theta\in\R$ we have 
\begin{align}
	\E_\lambda e^{\theta X_k}&=
	\E_\lambda e^{\theta X_k-\theta^2\langle X(k)\rangle+\theta^2\langle X(k)\rangle}\notag\\
	&\le \Bigl(\E_\lambda e^{2\theta X_k-2\theta^2\langle X(k)\rangle}\Bigr)^{1/2}
	\Bigl(\E_\lambda e^{2\theta^2\langle X(k)\rangle}\Bigr)^{1/2}.\label{expo-estimate-theta}
\end{align}
The second expectation in the right-most term of~\eqref{expo-estimate-theta} is bounded by a number not depending on~$k$, provided that $\theta^2$ is sufficiently small; see inequality~(3.6) in~\cite{shirikyan-ptrf2006}. In particular, the Novikov condition is satisfied for the martingale~$2\theta X_t(k)$, so that the first expectation in the right-most term of~\eqref{expo-estimate-theta} is equal to~$1$; see Proposition~5.12 in~\cite{KS1991}. We thus obtain the required bound~\eqref{ineq:expX_k} with $0<\theta\ll 1$ and a number~$\Theta$ not depending on~$k$.

	\smallskip
	{\it Step~2: Law of large numbers for~$V_n$\/}. Recalling that $X_k=M_k-M_{k-1}$, suppose we have established the relation 
	\begin{equation}\label{cond-expect}
		\E_\lambda\{X_k^2\,|\,\FF_{k-1}\}=g(u_{k-1}), \quad k\ge1,
	\end{equation}
	where the equality holds $\IP_\lambda$-almost everywhere, and $g:H\to\R$ is a continuous function satisfying the following condition for any $\delta>0$ and some number $\alpha\in(0,1)$ depending on~$\delta$:
	$$
\|g\|_{\alpha,\delta}:=\sup_{u\in H}\frac{|g(u)|}{w_\delta(u)}+\sup_{\substack{u,v\in H \\ u\ne v}}\frac{|g(u)-g(v)|}{\|u-v\|^\alpha\bigl(w_\delta(u)+w_\delta(v)\bigr)}<\infty.
	$$
	In this case, as is proved in Step~3 of~\cite[Section~3.2]{shirikyan-ptrf2006}, for any integer $q\ge1$ we have 
	\begin{equation}\label{LLN-rate}
		\E_\lambda\bigl|n^{-1}V_n^2 - \langle g,\mu_\nu\rangle\bigr|^{2q}
		\le C_{q,\nu}\langle w_\varkappa,\lambda\rangle n^{-q}, \quad n\ge1, 
	\end{equation}
	where $\mu_\nu\in\PP(H)$ is the unique stationary measure of~\eqref{NS}--\eqref{external-force}, and $C_{q,\nu}>0$ is a number not depending on~$n$ and~$\lambda$. Let us assume, in addition, that 
	\begin{equation}\label{sigma-nu}
		\sigma_\nu^2:=\langle g,\mu_\nu\rangle>0.
	\end{equation}
	Combining inequalities~\eqref{clt-lln} and~\eqref{LLN-rate}, in which  $\sigma=\sigma_\nu$, $\e>0$ is arbitrary, and $q=(16\e)^{-1}$, we derive~\eqref{clt-for-Mk}. Thus, it remains to establish~\eqref{cond-expect} and~\eqref{sigma-nu}.
	
	In view of the Markov property, relation~\eqref{cond-expect} holds with 
	\begin{equation*}
		g(v)=\E_vX_1^2=\E_v\biggl(\int_0^1\bigl(u(t),\dd\zeta(t)\bigr)\biggr)^2
		=\sum_{j=1}^\infty b_j^2\int_0^1\E_v u_j(t)^2\dd t,
	\end{equation*}
	where $v\in H$.	Denoting $f_j(v)=(v,e_j)^2$ and writing $\{\PPPP_t\}_{t\ge0}$ for the Markov semigroup associated with~\eqref{NS-reduced}, we see that 
	\begin{equation}\label{g-v}
		g(v)=\sum_{j=1}^\infty b_j^2\int_0^1(\PPPP_tf_j)(v)\,\dd t.
	\end{equation}
	Applying Lemma~3.2 in~\cite{shirikyan-ptrf2006}, we see that $\|\PPPP_tf_j\|_{\alpha,\delta}\le C$ for any $t\in[0,1]$ and $j\ge1$. Taking the norm in~\eqref{g-v}, we see that $\|g\|_{\alpha,\delta}\le BC$. Finally, integrating~\eqref{g-v} against $\mu_\nu$ and using the stationarity of~$\mu_\nu$, we derive
	$$
	\langle g,\mu_\nu\rangle=\sum_{j=1}^\infty b_j^2\int_H v_j^2\mu_\nu(\dd v).
	$$
	Since all the numbers~$b_j$ are positive and~$\mu_\nu$ cannot be concentrated at zero, we conclude that~\eqref{sigma-nu} holds. This completes the proof of Theorem~\ref{thm:asympto}.\qed
	
\addcontentsline{toc}{section}{Bibliography}
\def\cprime{$'$} \def\cprime{$'$}
  \def\polhk#1{\setbox0=\hbox{#1}{\ooalign{\hidewidth
  \lower1.5ex\hbox{`}\hidewidth\crcr\unhbox0}}}
  \def\polhk#1{\setbox0=\hbox{#1}{\ooalign{\hidewidth
  \lower1.5ex\hbox{`}\hidewidth\crcr\unhbox0}}}
  \def\polhk#1{\setbox0=\hbox{#1}{\ooalign{\hidewidth
  \lower1.5ex\hbox{`}\hidewidth\crcr\unhbox0}}} \def\cprime{$'$}
  \def\polhk#1{\setbox0=\hbox{#1}{\ooalign{\hidewidth
  \lower1.5ex\hbox{`}\hidewidth\crcr\unhbox0}}} \def\cprime{$'$}
  \def\cprime{$'$} \def\cprime{$'$} \def\cprime{$'$}
\providecommand{\bysame}{\leavevmode\hbox to3em{\hrulefill}\thinspace}
\providecommand{\MR}{\relax\ifhmode\unskip\space\fi MR }
\providecommand{\MRhref}[2]{%
  \href{http://www.ams.org/mathscinet-getitem?mr=#1}{#2}
}
\providecommand{\href}[2]{#2}

\end{document}